\documentclass[a4paper,12pt]{amsart}
\usepackage{graphicx,amsmath,amsfonts,amssymb,amsthm,mathrsfs}

\newcommand{\FF}{{\mathbb F}}
\newcommand{\ZZ}{{\mathbb Z}}

\newcommand{\QQ}{{\mathbb Q}}

\def\Fq2{{\mathbb F}_{q^2}}
\def\Fp2{{\mathbb F}_{p^2}}

\def\PP{{\mathbb P}}

\def\cD{{\mathcal D}}
\def\cE{{\mathcal E}}
\def\cC{{\mathcal C}}

\def\cX{{\mathcal X}}
\def\cJ{{\mathcal J}}
\def\cA{{\mathcal A}}
\def\cY{{\mathcal Y}}
\numberwithin{equation}{section}
\theoremstyle{plain}
\newtheorem{thm}{Theorem}
\numberwithin{thm}{section}
\newtheorem{lem}[thm]{Lemma}

\newtheorem{pro}[thm]{Proposition}

\newtheorem{remark}[thm]{Remark}
\newtheorem{con}[thm]{Conjecture}

\def\blfootnote{\xdef\@thefnmark{}\@footnotetext}
\begin{document}
\title{On certain maximal hyperelliptic curves related to Chebyshev polynomials}

\author{Saeed Tafazolian and Jaap Top
   }
\date{}
\address{University of Campinas (UNICAMP)\\
Institute of Mathematics, Statistics and Computer Science (IMECC)\\
Rua S\'{e}rgio Buarque de Holanda, 651, Cidade Universit\'{a}ria\\
13083-859, Campinas, SP, Brazil}
\address{Johan Bernoulli Institute for Mathematics and Computer Science\\
Nijenborgh~9\\9747 AG Groningen\\ the Netherlands}
\email{tafazolian@ime.unicamp.br}
\email{j.top@rug.nl}


   \begin{abstract}
 We study hyperelliptic curves arising from Chebyshev polynomials. The aim of this paper is to characterize the pairs $(q,d)$ such that the
 hyperelliptic curve $\cC$ over a
finite field $\FF_{q^2}$ given by $y^2 = \varphi_{d}(x)$  is maximal
over the finite field $\FF_{q^2}$ of cardinality $q^2$.
Here $\varphi_{d}(x)$ denotes the Chebyshev polynomial of degree $d$.
The same question is studied for the curves given by
 $y^2=(x\pm 2) \varphi_{d}(x)$, and also for
$y^2=(x^2-4)\varphi_d(x)$. Our results generalize some of the statements in \cite{KJW}.
   \end{abstract}



\maketitle

\section{Introduction}

Let $p$ be an odd prime number, let $q$ be a power of $p$, and denote by $\FF_{q^2}$ the finite field with $q^2$ elements.
 Let $\mathcal{C}$ be a curve (complete, smooth, and
 geometrically irreducible) of
 genus $g \geq 0$ over the finite field $\FF_{q^2}$.  We call the curve $\cC$ maximal over  $\FF_{q^2}$ if the number of rational points of $\cC$ over  $\FF_{q^2}$ attains the upper bound of Hasse-Weil, i.e,
$$\#\mathcal{C}(\FF _{q^2})  = 1 + q^2 + 2g q.$$
Not only have maximal curves several intrinsic geometrical properties, but also they have been investigated in connection with Coding Theory:
in some cases the best known linear codes over finite fields of square order are obtained as one-point AG-codes from maximal curves.

In this note we consider  hyperelliptic curves given by
one of the equations $y^2 = \varphi_{d}(x)$  or
$y^2=(x\pm 2) \varphi_{d}(x)$ or
$y^2=(x^2-4)\varphi_d(x)$ over $\FF_{q^2}$. Here $\varphi_d(x)$ denotes
the Chebychev polynomial of degree $d$ over
$\FF_p\subset\FF_{q^2}$: recall that this is the reduction modulo $p$ of the unique polynomial $\phi(X) \in \ZZ[X]$ such that
\[
x^d+x^{-d}=\phi(x+x^{-1})
\]
in $\ZZ[x,x^{-1}]$.

\begin{remark}
\rm{Note that  $\varphi_d(x)=D_d(x,1)$ with $D_d$ the $d$-th Dickson polynomial of the first kind with parameter $1$, defined recursively by
\[D_n(x,1) =xD_{n-1}(x,1)-D_{n-2}(x,1)\]
for $n \geq 2$, and $D_0(x,1) = 2$ and $D_1(x,1) =x$.
 Dickson polynomials are related to the classical Chebyshev polynomials $T_n(x)$,  defined for each integer $ n \geq 0$ by $T_n(x)= \cos ( n ~ \arccos x)$; indeed we have that $D_n(x,1)=2T_n(x/2).$ Because of this connection, these Dickson polynomials are  also called Chebyshev polynomials (see \cite[Page 355]{LN}), a convention we follow here.
}
\end{remark}

In Lemma~\ref{sep} we describe the pairs $(q,d)$ such that
$\varphi_d(x)$ is a separable polynomial (over $\FF_q$).
Our main goal is to study the problem,
for which pairs $(q,d)$ the curve in question; so, given by one of the equations
$y^2=\varphi(x)$ or $y^2=(x\pm 2)\varphi_d(x)$ or $y^2=(x^2-4)\varphi_d(x)$, is maximal over $\FF_{q^2}$. 
Throughout, when we write ``curve with (affine) equation $y^2=f(x)$'' or even $C\colon y^2=f(x)$ we mean that we
consider the smooth, complete curve birational (over the ground field) to the curve given by the affine equation.We have the following results.

\begin{thm}\label{maineven}
  Let $d>0$ be an even integer and let $q$ be a prime power
  with $\gcd(q,d)=1$. Then the hyperelliptic curve $\cC$
  given by 
\[y^2=(x+2) \varphi_{d}(x)\] is maximal over $\FF_{q^2}$ if and only if either $q \equiv -1~(\bmod~4d) ~\mbox{or} ~ q\equiv 2d+1 ~ (\bmod~4d).$
\end{thm}

\begin{remark}\label{rem1.2}{\rm
The definition of the polynomials $\varphi_d$ implies
that $\varphi_d(-x)=(-1)^d\varphi_d(x)$.
As a consequence, for $d$ even and $q$ odd, using
a primitive $4$-th root of unity $i\in\FF_{q^2}$ one
obtains an isomorphism over $\FF_{q^2}$ given by
$(u,v)\mapsto (-u,iv)$ from the curve $C$ described above,
to the curve with equation $y^2=(x-2)\varphi_{d}(x)$.
Hence for this curve, the same maximality criteria over
$\FF_{q^2}$ hold as those described in Theorem~\ref{maineven}.

Another property that is immediate from the definition of
the polynomials $\varphi_d(x)$ is that if $d=a\cdot b$ for
positive integers $a,b$, then $\varphi_d(x)=
\varphi_a(\varphi_b(x))$.
Applying this in the situation of Theorem~\ref{maineven}
with $a=d/2$ and $b=2$, one obtains $\varphi_d(x)=
\varphi_{d/2}(x^2-2)$.
Writing the equation for $\cC$ as
$y^2=(x+2)\varphi_{d/2}(x^2-2)$, it already appears
in \cite[Proposition~3]{TTV}. In fact this will be
used in Section~\ref{five}.
}
\end{remark}

The analog of Theorem~\ref{maineven}
for odd $d$ is as follows.

\begin{thm}\label{x2mainodd}
Let $d\geq 1$ be an odd integer and let $q$ be a prime power
with $\gcd(q,2d)=1$. Then the hyperelliptic curve $\cC$
given by
\[
y^2=(x+2)\varphi_d(x)
\]
is maximal over $\FF_{q^2}$ if and only if $q\equiv -1
~(\bmod~2d)$.
\end{thm}

For the hyperelliptic curve given by $y^2=\varphi_d(x)$ our strongest
results are obtained in the case that $d$ is even:

\begin{thm}\label{evencheb}
Suppose $d>0$ is an even integer and $q$ is a prime power
with $\gcd(q,d)=1$.
Then the following statements are equivalent.
\begin{itemize}
\item[{\rm (i)}] the hyperelliptic curve $\cC_1\colon y^2= (x^2-4)\varphi_{d}(x)$ is maximal over $\FF_{q^2}$;
\item[{\rm (ii)}] $q\equiv -1~(\bmod~4)$ and the hyperelliptic curve $\cC\colon  y^2=\varphi_d(x)$
 is maximal over $\FF_{q^2}$;
 \item[{\rm (iii)}] $q\equiv -1~(\bmod~2d)$.
\end{itemize}
\end{thm}

For odd $d>0$  we have the following somewhat weaker result.
\begin{thm}\label{mainodd}
  Let $d>0$ be an odd integer and let $q$ be a
  prime power. Assume that $q$ is coprime to $2d$.
  If $q \equiv -1(\bmod~4d)$ or $q\equiv 2d+1 (\bmod~4d)$, then the curve $\cC\colon y^2= \varphi_{d}(x)$ is maximal over $\FF_{q^2}$,
  and so is the curve $\cC_1\colon  y^2=(x^2-4)\varphi_d(x)$.
 If both $\cC$ and $\cC_1$ are maximal over
 $\FF_{q^2}$, then either
 $q \equiv -1(\bmod~4d)$ or $q\equiv 2d+1 (\mbox{mod}~4d)$.
\end{thm}

Based on considering small cases and experiments
using Magma (see also the discussion in Remark~\ref{rem4.3} and the special case based on a result of Kohel and Smith
which we discuss in Remark~\ref{kohelsmithresult}), we in fact have a stronger expectation for odd $d>0$:
\begin{con}\label{conj}
For any prime power $q$ and any odd $d>0$ with $\gcd(q,2d)=1$,
the following statements are equivalent.
\begin{itemize}
\item[{\rm (i)}] the hyperelliptic curve $\cC_1\colon y^2= (x^2-4)\varphi_{d}(x)$ is maximal over $\FF_{q^2}$;
\item[{\rm (ii)}] $q\equiv -1(\bmod~4)$ and the hypereliptic curve $\cC\colon  y^2=\varphi_d(x)$
 is maximal over $\FF_{q^2}$.
\end{itemize}
\end{con}
 Clearly if Conjecture~\ref{conj} holds then a more complete
 and simple criterion follows (using Theorem~\ref{mainodd} and
 similar to  Theorem~\ref{evencheb}).

In Sections~\ref{two} and \ref{three}
some necessary background is recalled and a
general necessary condition on the characteristic
is shown (Proposition~\ref{p2.3}) in order for
a hyperelliptic curve with
equation $y^2=xg(x^2)$ over $\FF_{q^2}$ (of positive genus) to be maximal.
Section~\ref{proofs} contains the proofs
of most results announced in this introduction.
In Section~\ref{five} we prove Theorem~\ref{evencheb}
and discuss Conjecture~\ref{conj}. We finish with a small
application/illustration of Complex Multiplication theory (Proposition~\ref{CMresult}).

 \section{Preliminaries}\label{two}
  The \emph{zeta function} of a curve $\mathcal{C}$
  over a finite field $k$ of cardinality $q$ is a
 rational function of the form
\[Z(\mathcal{C}/k)=\dfrac{L(t)}{(1-t)(1-qt)},\]
\noindent where $L(t) \in \ZZ[t]$ is a polynomial of degree $2g=2\cdot\mbox{genus}(\cC)$
with integral coefficients (see \cite[Chapter V]{St}).
 We call this polynomial the $L$-\textit{polynomial} of $\mathcal{C}$ over $k$.

\vspace{.2cm}

 We recall the following fact about maximal curves which can be deduced by extending the argument on p. 182 of \cite{St}.

\begin{pro}\label{p2.1}
Suppose $q$ is a square. For a smooth projective curve $\mathcal{C}$
of genus $g$, defined over $ k=\FF_{q}$, the following conditions
are equivalent:
\begin{itemize}
\item $\mathcal{C}$ is maximal   over $\FF_{q}$.
 \item $L(t)= (1+\sqrt{q}t)^{2g} $.
\end{itemize}
\end{pro}

\vspace{.2cm}

 A common method to construct (explicit) maximal curves is via the following remark  which although commonly attributed to J-P. Serre (cf. Lachaud \cite{La}), is implicitly already contained in Tate's seminal paper \cite{Tate}:

\begin{remark}\label{r1}
\rm{Given a non-constant morphism $f\colon\mathcal{C} \longrightarrow
\mathcal{D}$ defined over
the finite field $k$, the $L$-polynomial of $\mathcal{D}$
over $k$ divides
the one of $\mathcal{C}$ over $k$. Hence a subcover $\mathcal{D}$ over $\FF_{q^2}$ of a
maximal curve $\mathcal{C}$ over $\FF_{q^2}$ is also maximal.}

Many examples of maximal curves have been found in this way starting from `standard' known ones.
In various cases this is done including explicit equations for the subcover, in other cases by
merely identifying appropriate subfields (and the genus of the corresponding curve) of
a function field $\FF_{q^2}(\mathcal{C})$ of a maximal $\mathcal{C}/\FF_{q^2}$.
From the abundant literature on this, we mention \cite{GSX}, \cite{AQ}, \cite{CO}, \cite{ABB},
\cite{FF2}, \cite{JPAA}, \cite{SFN}, \cite{GQZ}, \cite{GMQZ}, \cite{BM}, \cite{ASF}.

In the present paper we work in some sense `the other way around': the curves we study are
indeed subcovers $\mathcal{D}$ (by an morphism of degree $2$) of  curves $\mathcal{C}$ for which maximality
properties are precisely known. By identifying the $L$-polynomial of $\mathcal{C}$ essentially in terms of
that of $\mathcal{D}$ in the cases at hand, which is done by `understanding' up to isogeny
the Jacobian variety of $\mathcal{C}$ in terms of that of $\mathcal{D}$,
we obtain necessary (and not only sufficient) maximality criteria for $\mathcal{D}$. 
  \end{remark}

The following result yields a necessary condition
for maximality of a special type of hyperelliptic
curves.
\begin{pro}\label{p2.3}
Let $q=p^n$ be the cardinality of a finite field
$\FF_q$ of characteristic $p>2$.
Suppose $g(x)\in\FF_q[x]$ is separable 
of degree $d\geq 1$, and $g(0)\neq 0$. Let $\cC$ be the hyperelliptic curve
over $\FF_q$ with equation
$y^2=xg(x^2)$.\\
If the Jacobian of $\cC$ is supersingular, then $p\equiv 3~(\bmod~4)$.\\
As a consequence, if $\cC$ is maximal over $\FF_{q^2}$, then
$p\equiv 3~(\bmod~4)$.
\end{pro}
\begin{proof}
The assumptions imply that $\cC$ is a curve
of genus $d\geq 1$.
Let $i$ be a primitive $4$-th root of unity in
some extension of $\FF_q$. The curve $\cC$ admits the automorphism $\iota$ given by $\iota(x,y)=(-x,iy)$.
The action of $\iota$ on the vector space of regular
$1$-forms on $\cC$ is diagonalizable, and has as
eigenvalues $\pm i$.

We claim that maximality of $\cC$ over $\FF_{q^2}$
implies that the characteristic $p$ of $\FF_q$
satisfies $p\equiv 3(\bmod~4)$. Indeed, if
$p\equiv 1(\bmod~4)$ then take integers $a,b$
such that $p=a^2+b^2$.
As endomorphisms of $\cJ=\mbox{Jac}(\cC)$ this yields
a factorization $p=(a+b\iota)(a-b\iota)$.
Since multiplication by $p$ is inseparable, at least one
of the endomorphisms $a\pm b\iota$ is inseparable as well.
However, it is not possible that both are inseparable since
that would imply the sum $2a$ to be inseparable as well,
which clearly is not the case.
This means that after changing the sign of $b$ if necessary,
we have that $a+b\iota$ is separable. Hence its kernel
$\cJ[a+b\iota](\overline{\FF_q})$ is a nontrivial
subgroup of the $p$-torsion of $\cJ$, which shows that $\cJ$ is not supersingular.

Since the $p$-torsion of the Jacobian of any maximal
curve over $\FF_{q^2}$ is trivial,  $\cC$ cannot be maximal over any finite field
of characteristic $\equiv 1(\bmod~4)$.
So we have $p\equiv 3(\bmod~4)$.
\end{proof}
\begin{remark}{\rm
The assumption that a curve $\cC$ of genus $g$ is maximal over $\FF_{q^2}$
implies that the $L$-polynomial of $\cC$ over $\FF_q$
(which has as zeros square roots of the zeros of the
$L$-polynomial of $\cC$ over $\FF_{q^2}$) must be
$(1+qt^2)^{2g}$. In the situation described in
Proposition~\ref{p2.3} this means that if $q$
is a square, then the quartic
twist of $\cC$ over $\FF_q$ corresponding to the
cocycle $F_q\mapsto\iota$ (with $F_q$ the $q$-th power
Frobenius) has $L$-polynomial $(1-qt^2)^{2g}$.
In case $q$ is not a square, the analogous cocycle
results in a twist that has (again) $L$-polynomial $(1+qt^2)^{2g}$.
}\end{remark}

We finish this section with a preliminary result generalizing parts of
\cite[Theorem~6.1(b)]{GS} and \cite[Theorem~7.2(a)]{ASF}
 (in fact it is based on essentially the same ideas
 already present in \cite{GS}).
 \begin{lem}\label{sep}
 For $d\in\ZZ_{>0}$ and $q$ a prime power,
 the Chebyshev polynomial $\varphi_d$ considered
 over the finite field $\FF_q$ of cardinality $q$
 is separable if and only if $\gcd(q,2d)=1$ or $d=1$.
 \end{lem}
 \begin{proof}
 Consider the morphism $\alpha\colon {\mathbb P}^1\to{\mathbb P}^1 $ given (in terms of local
 coordinates) by $\alpha(x)=x^d+x^{-d}$.
 One factors $\alpha=\beta\circ \gamma$ with
 $\gamma\colon\PP^1\to\PP^1$ given by
 $\gamma(x)=x^d$ and
 $\beta\colon\PP^1\to\PP^1$ by $\beta(x)=x+x^{-1}$.
 Regarding $\varphi_d$ as the morphism $\PP^1\to\PP^1$
 given by $x\mapsto\varphi_d(x)$, by
 definition $\alpha=\beta\circ\gamma=\varphi_d\circ\beta$. We study separability of the
 {\em polynomial} $\varphi_d$, which means we examine
 whether the {\em morphism} $\varphi_d$ is separable
 and moreover has no ramification points over
 $0\in\PP^1$. To this end,  first consider separability
 (and ramification) of the two morphisms $\gamma$ and $\beta$.

 Clearly $\beta$ is a separable morphism of degree $2$, in every characteristic. It is only ramified in
 $\pm 1$, and this is one point in characteristic $2$
 and two points in every other characteristic.

The morphism $\gamma$ is inseparable precisely when
$\gcd(q,d)\neq 1$. If this holds then also
$\alpha=\beta\circ\gamma$ is inseparable. As
a consequence, so is $\varphi_d$ since
$\alpha=\gamma\circ \varphi_d$ and $\gamma$ is separable. So
\[\gcd(q,d)\neq 1\quad\Rightarrow\quad \mbox{the~polynomial}
\;\; \varphi_d\;\; \mbox{is~inseparable~over}\;\;\FF_q.\]
Next, assume $\gcd(q,d)=1$ so that $\gamma$
is separable (over $\FF_q$). Then $\alpha$ and hence
$\varphi_d$ are separable morphisms as well.
To obtain the ramification points of $\varphi_d$ in
this case, we compute the ramification of
$\alpha=\beta\circ\gamma$. First consider the case
that $q$ is odd. Then $\beta$ is only ramified
in $\pm 1$ (both points with ramification index
$e_{\pm1}=2$ and $\beta(\pm 1)=\pm 2$).
Moreover $\gamma$ is only ramified in $0$ and in
$\infty$ (both with ramification index $d$)
and $\gamma^{-1}(\pm 1)$ consists of the
$2d$ pairwise distinct solutions of $x^{2d}=1$.
Since $\gamma^{-1}(0)=0$ and $\gamma^{-1}(\infty)
=\infty$, the conclusion is that the total map
$\alpha$ is ramified only in the following points:
$\{0,\infty\}$, each with ramification index $d$,
and in the $2d$-th roots of unity, each with
ramification index $2$. Moreover the image of these
points under $\alpha$ is $\{\infty,\pm 2\}$.
Since $0\not\in\{\infty,\pm 2\}$ and
$\alpha=\varphi_d\circ\beta$, one concludes
\[q\;\;\mbox{is~odd~and}\;\;\gcd(q,d)=1
\quad\Rightarrow\quad \mbox{the~polynomial}
\;\; \varphi_d\;\; \mbox{is~separable~over}\;\;\FF_q.
\]
Now consider the case $2|q$ and $\gcd(q,d)=1$.
This implies that the map $\alpha$ is separable
over $\FF_q$. As in the previous case, the ramification
of $\alpha$ is easily found using $\alpha=\beta\circ\gamma$. Now $\beta$ is only ramified at $1$,
with $\beta(1)=0$ (ramification index $2$).
We conclude that $\alpha$ is ramified only in the
following points: $\{0,\infty\}$, each with ramification
index $d$, and in the $d$-th roots of unity,
each with ramification index $2$.
The image of these points under $\alpha$ is
$\{\infty,0\}$.
As $\alpha^{-1}(0)$ consists of the $d$-th roots of unity and only $1\in\alpha^{-1}(0)$ is a ramification
point of $\beta$, the decomposition $\alpha=
\varphi_d\circ\beta$ shows that whenever $d>1$
then $\varphi_d\colon\PP^1\to\PP^1$ is ramified in
some points over $0$ (namely, in
$\zeta+\zeta^{-1}$ with $\zeta\neq 1$ satisfying
$\zeta^d=1$). We showed:
\[
2|q\;\;\mbox{and}\;\;\gcd(q,d)=1\;\;\mbox{and}
\;\; d>1
\quad\Rightarrow\quad \mbox{the~polynomial}
\;\; \varphi_d\;\; \mbox{is~inseparable~over}\;\;\FF_q.
\]
Since the case $d=1$ (so $\varphi_d(x)=x$) is trivial,
the lemma follows.
 \end{proof}

\section{The curves $y^2=x^{2d+1}+x$ and $y^2=x^{2d}+1$}\label{three}

Let $d\geq 1$ be an integer, and let $q$ be a
prime power such that $\gcd(q,2d)=1$.
We consider the complete non-singular curve $\cX$
over $\Fq2$ birational to the plane affine curve given by
  $$
  y^2=x^{2d+1}+x\, .
  $$
The condition on the pair $(q,d)$ implies that $\cX$
has genus $d$.

\vspace{.3cm}

The following result is crucial for us  (see \cite[Theorem 1]{FF2}).

\vspace{.2cm}

  \begin{thm}\label{main}
   The smooth complete hyperelliptic curve $\cX$ given by
    \[y^2=x^{2d+1}+x\]
   is maximal over $\FF_{q^2}$ if and only if either
   $q \equiv -1~(\bmod~4d)~\mbox{or}~
   q\equiv 2d+1~(\bmod~4d).$
  \end{thm}

Now let $\cY$ be the complete non-singular curve
over $\FF_q$ given by $y^2=x^{2d}+1$.
Note that the condition $\gcd(q,2d)=1$ implies
that $\cY$ has genus $d-1$.

\vspace{.3cm}

One more result which will be used in our proofs is recalled from \cite{JPAA}:

\vspace{.2cm}

\begin{thm}\label{main2}
The smooth complete hyperelliptic curve $\cY\colon y^2=x^{2d}+1$ is maximal over $\FF_{q^2}$ if and only if
$q\equiv -1 (\bmod~2d)$.
\end{thm}

\section{Hyperelliptic curves from Chebyshev
polynomials}\label{proofs}

In this section we prove Theorems~\ref{maineven}, \ref{x2mainodd}, \ref{mainodd}, and we present and prove
some preliminary results which will be used in the proof of Theorem~\ref{evencheb}.

\subsection*{Case ${d}$ even and ${v^2=(u+2)\varphi_d(u)}$}

\begin{proof} (of Theorem~\ref{maineven}).
Take $d>0$ an even integer, and let $q$ be a prime power with $\gcd(d,q)=1$.  We will
show that the curve $\cC$ with affine
equation
\[v^2=(u+2) \varphi_{d}(u)\] is maximal over $\FF_{q^2}$ if and only if the curve $\cX$ introduced in
Section~\ref{three} (with equation $y^2=x^{2d+1}+x$)
is maximal over $\FF_{q^2}$. Theorem~\ref{maineven} is then a
consequence of Theorem~\ref{main}.

The main idea is to decompose the Jacobian
variety $\cJ(\cX)$ up to isogeny over $\FF_{q^2}$.
Let $\tau\in\mbox{Aut}(\cX)$ be the involution
given by $\tau(x,y)=(1/x,y/x^{d+1})$.
    The quotient of $\cX$ by $\tau$ is the curve $\cC=\cX/<\tau>$ with  equation
  $$v^2=(u+2) \varphi_{d}(u);$$
indeed, the functions $u=x+1/x$ and $v=y(x+1)x^{-1-d/2}$ generate the subfield of $\tau$-invariants
in the function field of $\cX$, as is seen as follows.
Write $\FF_p(x,y)$ for the function field of $\cX$
over the prime field $\FF_p$ of $\FF_q$.
We have the inclusions of fields (where the numbers
describe the degree of the given extensions)
\[
\begin{array}{ccc}
\FF_p(x,y) & {\supset} & \FF_p(u,v)\\
\stackrel{ \scriptstyle \phantom{2}}{\cup{\scriptstyle 2}} && \cup{\scriptstyle 2} \\
\FF_p(x) & \stackrel{{\scriptstyle 2}}{\supset}
& \FF_p(u)
\end{array}
\]
Since $[\FF_p(x,y):\FF_p(x,y)^{<\tau>}]=2$ and
$u,v\in\FF_p(x,y)^{<\tau>}$, one has $\FF_p(x,y)^{<\tau>}=\FF_p(u,v)$.
Moreover, $u,v$ satisfy
\[v^2=(x^{2d+1}+x)(x+2+x^{-1})x^{-d-1}=(u+2)\varphi_d(u).\]
  We have the  basis
   $$\{ \omega_j:=\frac{x^{j-1} dx}{y} \;|\; 1 \leq j \leq d  \}$$
    for the space of regular differentials on $\cX$. A basis for the differentials invariant under $\tau$ is  $$\{ \omega_j-\omega_{d-j+1}\;|\; 1 \leq j \leq d/2   \},$$
  which also generate the pull-backs of the regular differentials on $\cC$ (note that since we assume
$\gcd(q,d)=1$, Lemma~\ref{sep} implies that
$\varphi_d$ is separable over $\FF_q$. Also,
$\varphi_d(-2)=\varphi_d((-1)+(-1))=(-1)^d+(-1)^d=2\neq 0$, so $\cC$ has genus $d/2$).

  Let $\iota$ be the hyperelliptic involution on $\cX$,
  so $\iota(x,y)=(x,-y)$.
  The quotient of $\cX$ by $\tau\iota$ (this map is an
  involution defined over the prime field) is the curve $\cC_1=\cX/<\tau\iota>$ with equation
  $$\eta^2=(\xi -2) \varphi_{d}(\xi);$$
  indeed, the invariants under $\rho$ in the function field
  of $\cX$ are generated by $\xi:=x+x^{-1}$ and
  $\eta:=y(x-1)x^{-1-d/2}$. These functions satisfy
  \[\eta^2=(x^{2d+1}+x)(x-2-x^{-1})x^{-d-1}=
  (\xi-2)\varphi_d(\xi).\]
A basis for the differentials invariant under $\tau\iota$ is
   $$\{ \omega_j+\omega_{d-j+1}| 1 \leq j \leq d/2  \},$$
  which also generate the pull-backs to $\cX$ of the regular differentials on $\cC_1$.

  Fixing a primitive $4$-th root of unity $i\in\FF_{q^2}$,
  the map
  $(u,v)\mapsto (-u, iv)$ yields an isomorphism
  $\cC\cong \cC_1$ defined over $\FF_{q^2}$. The discussion above
  shows, with $\sim$ denoting isogeny defined over
  $\FF_{q^2}$, that
  \[\cJ(\cX) \sim \cJ(\cC)\times\cJ(\cC_1)\cong \cJ(\cC)^2.\]
 As a consequence $L_{\cX}(t)=L_{\cC}(t)^2$
 with $L$ denoting an $L$-polynomial over $\FF_{q^2}$. Now
 Proposition \ref{p2.1} implies that the curve $\cX$ is maximal if and only if the curve $\cC$ is maximal. This completes the proof.

\end{proof}

\begin{remark}
\rm{Theorem~\ref{maineven} generalizes a part of \cite[Proposition 6]{KJW}.}
The decomposition up to isogeny of the Jacobian variety $\cJ(\cX)$ as a product of Jacobians of quotient
curves, can also be obtained using results of Kani and Rosen \cite{Kani-Rosen}.
There are various examples in the literature illustrating this technique; we refer to
\cite[\S~3.1.1]{Paulhus} and
\cite[p.~36]{Soomro} for situations very similar to the ones discussed in the present paper.
\end{remark}

\subsection*{Case $d$ odd and $v^2=(u+2)\varphi_d(u)$}
\begin{proof} (of Theorem~\ref{x2mainodd}).
This is very similar to the proof of Theorem~\ref{maineven}.
Take $d>0$ an odd integer, and let $q$ be a prime power with
$\gcd(q,2d)=1$.
We will
show that the curve $\cC$ with affine
equation
\[v^2=(u+2) \varphi_{d}(u)\] 
is maximal over $\FF_{q^2}$ if and only if the curve $\cY\colon y^2=x^{2d}+1$ is maximal over $\FF_{q^2}$. Theorem~\ref{x2mainodd} is then a
consequence of Theorem~\ref{main2}.

Let $\sigma$ be the involution on $\cY$ defined by $\sigma(x,y)=(1/x,y/x^d)$.
The quotient of $\cY$ by $\sigma$ is the hyperelliptic curve
$\cC$. Indeed, the functions $u=x+x^{-1}$ and $v=y(1+x)x^{-(d+1)/2}$
generate the field of functions invariant under $\sigma$,
and one computes
\[ v^2=y^2\cdot x^{-(d+1)}\cdot(x+1)^2=(x^d+x^{-d})(x+2+x^{-1})=(u+2)\varphi_d(u).\]
Multiplying $\sigma$ by the hyperelliptic involution on $\cY$
one obtains another quotient curve which we denote by $\cC_1$. The invariant functions
under the new involution are generated by $u=x+x^{-1}$ and $w=y(1-x)x^{-(d+1)/2}$.
They satisfy $w^2=(u-2)\varphi_d(u)$.
The map $(u,w)\mapsto (-u,w)$ defines an isomorphism $\cC_1\cong \cC$.

Analogous to the previous proof one concludes
$L_{\cY}(t)=L_{\cC}(t)\cdot L_{\cC_1}(t)=L_{\cC}(t)^2$, in this case for the $L$-polynomials
over $\FF_q$ as well as for those over $\FF_{q^2}$.
This implies the result.
\end{proof}
\subsection*{Case $d$ odd and $y^2=\varphi_d(x)$}
\begin{proof}(of Theorem~\ref{mainodd}).
Let $d\geq 1$ be an odd integer. Take a prime power
$q$ such that $\gcd(q,2d)=1$. We will consider curves
over (the prime field of) $\FF_{q^2}$.
Recall (see the proof of Theorem~\ref{maineven}) that the hyperelliptic curve $\cX$ 
with affine equation $y^2=x^{2d+1}+x$ admits the involution $\tau$ defined by $\tau(x,y)=(1/x, y/x^{d+1})$.
For odd $d$, the quotient of $ \cX$ by  $\tau$ is the hyperelliptic curve $\cC$ with equation $$y^2= \varphi_{d}(x);$$
indeed, a quotient map is given by
$$(x,y)  \mapsto (x+1/x, y/x^{(d+1)/2})$$
(compare \cite[Proposition~3]{TTV}).
Now if either
$q \equiv -1 ~(\bmod~4d)~\mbox{or}~q\equiv 2d+1~(\bmod~4d)$,
then by Theorem \ref{main} the curve $\cX$ is maximal over $\FF_{q^2}$ which implies that the curve $\cC$ is also maximal over $\FF_{q^2}$. This proves
the first assertion of Theorem~\ref{mainodd}.

To show the remaining parts, we will decompose
up to isogeny the Jacobian $\cJ(\cX)$ of the curve $\cX$.
With the basis $\omega_j:= x^{j-1} dx/y$ (for $1 \leq j\leq d$) for the regular differentials on $ \cX$, one checks that a basis for the differentials invariant under $\tau$ is
$$\omega_1 - \omega_d, \omega_2 - \omega_{d-1}, \cdots, \omega_{(d-1)/2} - \omega_{(d+3)/2},$$
which also generate the pull-backs of the regular differentials on $\cC$; note that by Lemma~\ref{sep}
the condition $\gcd(q,2d)=1$ implies that $\varphi_d$
is separable over $\FF_q$ hence $\cC$ has genus $(d-1)/2$.

Let $\iota$ be the hyperelliptic involution on $\cX$.  The quotient of $\cX$ by $\tau\iota$ (this automorphism
 has order $2$ and it is defined over the prime field) is the curve $\cC_1=\cX/<\tau\iota>$ with equation
  \[y^2=(x^2-4)\varphi_d(x);\]
  indeed, the functions $\xi:=x+1/x$ and
  $\eta:=\frac{y}{x^{(d+1)/2}}(x-\frac{1}{x})\in
  {\mathbb F}_q(\cX)$ are invariant under
the action of $\tau\iota$ and
$[{\mathbb F}_q(\cX):{\mathbb F}_q(\xi,\eta)]=2$. Hence
$\xi,\eta$ generate the function field of $\cC_1$.
We have
\[\eta^2=\frac{y^2}{x^{d+1}}(x^2-2+x^{-2})=(x^d+x^{-d})\left((x+\frac1x)^2-4\right)=(\xi^2-4)\varphi_d(\xi).\]
From this, the second assertion in Theorem~\ref{mainodd}
follows: namely, by Theorem~\ref{main} the
congruence condition on $q$ implies that
$\cX$ is maximal over $\FF_{q^2}$. Since $\cX$
covers $\cC_1$, the same is true for $\cC_1$ over
$\FF_{q^2}$.

Note that $\varphi_d(2)=\varphi_d(1+1)=1^d+1^d=2$
 and similarly $\varphi_d(-2)=-2$. Using
 Lemma~\ref{sep} this implies that in every
 characteristic coprime to $2d$ the polynomial
 $(x^2-4)\varphi_d(x)$ is separable.
 A basis for the differentials invariant under $\tau\iota$ is
 $\{  \omega_i+ \omega_{d-i+1}| 1 \leq i \leq (d+1)/2  \},$
  which also generate the pull-backs of the regular differentials on $\cC_1$.

Since the pull-backs of a basis of the regular differentials on $\cC$ together with the pull-backs of
a similar basis on $\cC_1$ yield a basis for the
regular differentials on $\cX$, one concludes that the Jacobian $\cJ(\cX)$ of $\cX$ is isogenous to a product $$\cJ(\cC) \times \cJ(\cC_1),$$
where $\cJ(\cC)$ and $\cJ(\cC_1)$ are the Jacobians of the curves $\cC$ and $\cC_1$, respectively.
This implies that $L_{\cX}(t)=L_{\cC}(t)\cdot L_{\cC_1}(t)$
(for $L$-polynomials over any extension of $\FF_q$).
Hence if both $\cC$ and $\cC_1$ are maximal over $\FF_{q^2}$
then so is $\cX$, which by Theorem~\ref{main} implies
that $q\equiv -1(\bmod~4d)$ or $q\equiv 2d+1(\bmod~4d)$.
This finishes the proof.
\end{proof}

\begin{remark}\label{Conjd=3}
\rm{The special case $d=3$ of the Theorem~\ref{mainodd}
is  a part of \cite[Proposition 4]{KJW}.}
In fact for $d=3$ one finds {({\it loc.\ sit.}) that $\cJ(\cC)$
is the elliptic curve $\cE_1$ with equation $y^2=x^3-3x$
and (up to isogeny)  $\cJ(\cC_1)$ is
a product $\cE_2\times \cE_3$ where
$\cE_2$ is the elliptic curve
with equation $y^2=x^3+x$ and $\cE_3$
is the one
with equation $y^2=x^3+108x$. These two elliptic curves $\cE_1$ and $\cE_2$
are isogenous over $\FF_{q^2}$ (for $q$ any prime power
with $\gcd(q,6)=1$). So in this case
maximality of any one of them over
$\FF_{q^2}$ is equivalent to
$q\equiv 3(\bmod~4)$ and to maximality
of any one of the curves $\cC$ or $\cC_1$
over $\FF_{q^2}$. In particular, Conjecture~\ref{conj}
holds for $d=3$.
}\end{remark}

\subsection*{Case $d$ even and $y^2=\varphi_d(x)$}
A preliminary result relying on an analogous
reasoning as above,
is the following which will be used in
the proof of Theorem~\ref{evencheb}.
\begin{lem}\label{mainevench}
Let $d>0$ be an even integer and let $q$ be a prime power.
Assume $\gcd(q,d)=1$. The next two statements
are equivalent.
\begin{itemize}
\item[{\rm (i)}] $q\equiv -1(\bmod~2d)$;
\item[{\rm (ii)}] the curve $\cC$ with affine equation
\[ v^2=\varphi_d(u)\]
is maximal over $\FF_{q^2}$, and so is the curve $\cC_1$
over $\FF_{q^2}$ given by
\[v^2=(u^2-4)\varphi_d(u).\]
\end{itemize}
\end{lem}
\begin{proof}
Take $d=2e$ for some integer $e>0$ and let $q$ be a prime
power with $\gcd(q,d)=1$.
The curve $\cY$ over $\FF_q$ with affine equation
$y^2=x^{2d}+1$ admits the involution $\sigma$ given by
$\sigma(x,y)=(1/x,y/x^d)$.
The functions in $\FF_q(\cY)$ which are invariant under
$\sigma$ are generated by $u=x+1/x$ and $v=\frac{y}{x^e}$.
We have
\[
v^2=x^{-d}(x^{2d}+1)=\varphi_d(x+\frac1x)=\varphi_d(u),
\]
so the quotient of $\cY$ by $\sigma$ is the curve $\cC$
given by $v^2=\varphi_d(u)$.

If $q\equiv -1\bmod~2d$ then by Theorem~\ref{main2}
the curve $\cY$ is maximal over $\FF_{q^2}$. Since
this curve covers $\cC$, it follows from Remark~\ref{r1}
that also $\cC$ is maximal over $\FF_{q^2}$.
This shows the first claim in Proposition~\ref{mainevench}.

For the second claim we use the product $\sigma'$
of $\sigma$ and the hyperelliptic involution on $\cY$,
so $\sigma'(x,y)=(1/x,-y/x^d)$.
The invariants in $\FF_q(\cY)$ under $\sigma'$ are
generated by $u=x+1/x$ and $w=\frac{y}{x^e}(x-\frac1x)$,
and they are related by
\[
w^2=(x^{2d}+1)x^{-d}(x^2-2+x^{-2})=(x^d+x^{-d})((x+\frac1x)^2-4)=(u^2-4)\varphi_d(u).
\]
So also $\cC_1\colon w^2=(u^2-4)\varphi_d(u)$ is covered by $\cY$.
Hence by Theorem~\ref{main2}, if $q\equiv -1~(\bmod~2d)$
then
$\cC_1$ is maximal over $\FF_{q^2}$, proving the second
claim in Proposition~\ref{mainevench}.

For the last claim, observe that analogous to the other results shown in this section we have that the Jacobian $\cJ(\cY)$
is isogenous over $\FF_q$ to the product
$\cJ(\cC)\times\cJ(\cC_1)$. Hence the $L$-polynomial of
$\cY$ over any extension of $\FF_q$ is the product of the
$L$-polynomials of $\cC$ and $\cC_1$ (over the same
extension). The remaining statement in Proposition~\ref{mainevench} is an immediate consequence of this.
\end{proof}

\section{Relating $y^2=\varphi_d(x)$ and $y^2=(x^2-4)\varphi_d(x)$}\label{five}
Here we prove Theorem~\ref{evencheb} and
we make some remarks concerning Conjecture~\ref{conj}.
The following lemma
turns out to be useful.
\begin{lem}\label{Lpols}
Let $d\geq 1$ be any integer and let $q$ be a prime power with
$\gcd(q,2d)=1$. Then the $L$-polynomial of the elliptic curve
over $\FF_q$ given by $y^2=x^3+x$ divides the $L$-polynomial
of the curve $\cC_1$ over $\FF_q$ with affine equation $y^2=(x^2-4)\varphi_d(x)$.
\end{lem}
\begin{proof}
First consider the case that $d$ is odd.
We use the notations from the proof
of Theorem~\ref{mainodd} and we let $\zeta$ in some
extension of $\FF_q$ be a primitive $4d$-th root of unity.
The curve $\cX$ admits an automorphism $\rho$ given by
$\rho(x,y)=(\zeta^2x,\zeta y)$.
 The quotient of $\cX$ by the group generated by $\rho^4$ is the elliptic curve $\cE$ given by $$y^2=x^3+x$$
 and an explicit quotient map is given by $$(x,y)\mapsto(x^d, x^{(d-1)/2}y).$$
 Note that although the elements of the group generated by $\rho^4$ may not be defined
 over $\FF_q$, the group is, which explains why the
 quotient curve and the map to it are defined over $\FF_q$.
A regular differential on $\cX$ invariant under $\rho^4$ is $ \omega_{(d+1)/2}=x^{(d-1)/2}dx/y$; observe that this
differential is a pull back of a regular differential on $\cC_1$.

As a consequence, the elliptic curve $\cE$
is up to isogeny contained in the
Jacobian $\cJ(\cC_1)$. This implies
the lemma for $d$ odd.

Now assume $d=2e$ is even. The curve $\cY\colon y^2=x^{4e}+1$ covers the given elliptic curve, with an explicit
covering map given by $(x,y)\mapsto (x^{2e},x^ey)$.
Note that $x^{e-1}\frac{dx}{y}$ is a pull-back to $\cY$ of a regular
differential on the elliptic curve.
The proof of Proposition~\ref{mainevench} shows that
$\cJ(\cC_1)$ is up to isogeny an abelian subvariety of $\cJ(\cY)$,
and the regular differentials on $\cY$ coming from $\cC_1$
are the ones invariant under the action of the automorphism
denoted $\sigma'$. As the differential
$x^{e-1}\frac{dx}{y}$ is invariant under $\sigma'$, it follows that
the elliptic curve is up to isogeny contained in $\cJ(\cC_1)$.
This implies the lemma for $d$ even.
\end{proof}
\begin{pro}\label{2mod4}
The analogue of Conjecture~\ref{conj} holds in the special case $d\equiv2~(\bmod~4)$.
\end{pro}
\begin{proof}
Write $d=2e$ with $e$ a positive, odd integer and let $q$ be a prime power
satisfying $\gcd(q,2e)=1$.
One decomposes, up to isogeny, the Jacobian $\cJ(\cC_1)$ of
the curve $\cC_1$ given by $y^2=(x^2-4)\varphi_{2e}(x)$ as follows.
Note that $\cC_1$ admits the involution $\alpha$ given by $\alpha(x,y)=(-x,y)$.
Since $\varphi_{2e}(x)=\varphi_e(\varphi_2(x))=\varphi_e(x^2-2)$,
the quotient by $\alpha$ is the curve $\cD$ with affine equation
$v^2=(t-4)\varphi_e(t-2)$ (with quotient map $(x,y)\mapsto (x^2,y)$).
Using the variables $v$ and $u:=t-2$, this equation becomes $v^2=(u-2)\varphi_e(u)$.

Using that the curve $\cD$ is isomorphic to the one with equation $v^2=(u+2)\varphi_e(u)$ (just change the sign of $u$ and use that $e$ is odd),
Theorem~\ref{x2mainodd} implies that if $\cC_1$ is maximal
over $\FF_{q^2}$, then so is $\cD$, and therefore $q\equiv -1(\bmod~2e)$.
From Lemma~\ref{Lpols}, the maximality of $\cC_1$ over $\FF_{q^2}$
implies maximality of the elliptic curve given by $y^2=x^3+x$ over $\FF_{q^2}$.
The latter maximality is equivalent to $q\equiv -1~(\bmod~4)$.

Using that $e$ is odd, one concludes that if $\cC_1$ is maximal over
$\FF_{q^2}$, then $q\equiv -1~(\bmod~4e)$. Hence Proposition~\ref{mainevench}
implies the implication $\mbox{\rm (i)}\Rightarrow\mbox{\rm (ii)}$ of
Conjecture~\ref{conj} in this case.

For the other implication, assume that $\cC\colon y^2=\varphi_{2e}(x)$
is maximal over $\FF_{q^2}$ and that $q\equiv -1~(\bmod~4)$.
Writing $\varphi_{2e}(x)=\varphi_e(x^2-2)$ it is clear that
the map $(x,y)\mapsto (x^2-2,xy)$ yields a nonconstant morphism from $\cC$
to the curve with equation $s^2=(t+2)\varphi_e(t)$.
Hence the latter curve is maximal over $\FF_{q^2}$, which by
Theorem~\ref{mainodd} implies $q\equiv -1~(\bmod~2e)$.
So again one concludes $q\equiv -1~(\bmod~4e)$,
and the maximality of $\cC_1$ over $\FF_{q^2}$ follows
from Proposition~\ref{mainevench}.
\end{proof}

Similar ideas allow one to obtain some results in the case $d\equiv 0~(\bmod~4)$:
\begin{pro}\label{0mod4}
Suppose the integer $d>0$ satisfies $4|d$, then the analogue of Conjecture~\ref{conj} holds.
\end{pro}
\begin{proof}
With notations as above, write $d=2e$.
The map $(x,y)\mapsto (x^2-2, y)$ shows that
$\cC_1$ covers the curve given by $v^2=(u-2)\varphi_{e}(u)$.
Hence as before, by Theorem~\ref{maineven} one concludes
that if $\cC_1$ is maximal over $\FF_{q^2}$,
then either $q\equiv -1~(\bmod~4e)$ or $q\equiv 2e+1~(\bmod~4e)$.

Similarly, the map $(x,y)\mapsto (x^2-2,xy)$ shows that
$\cC_1$ covers the curve with affine equation
$w^2=(u^2-4)\varphi_e(u)$.
Hence maximality of $\cC_1$ over $\FF_{q^2}$ implies
using Lemma~\ref{Lpols} that the elliptic curve with
equation $y^2=x^3+x$ is maximal over $\FF_{q^2}$.
As a consequence $q\equiv -1~(\bmod~4)$.

Combining the congruences for $q$ we conclude that
maximality implies $q\equiv -1(\bmod~2d)$.
Hence by Proposition~\ref{mainevench} the curve given
by $y^2=\varphi_d(x)$ is maximal over $\FF_{q^2}$,
which is what we wanted to show.

Vice versa, is $\cC$ with affine equation
$y^2=\varphi_d(x)$ maximal over $\FF_{q^2}$ and is moreover
$q\equiv -1~(\bmod~4)$,
then since $(x,y)\mapsto (x^2-2,xy)$ shows that $\cC$
covers the curve given by $v^2=(u+2)\varphi_e(u)$,
we conclude using Theorem~\ref{maineven} that
either $q\equiv -1~(\bmod~4e)$ or $q\equiv 2e+1~(\bmod~4e)$.
However, the additional condition on $q$ shows that
the latter congruence is impossible,
so one concludes $q\equiv -1~(\bmod~4e)$.
But then Proposition~\ref{mainevench} implies maximality
of $\cC_1$ over $\FF_{q^2}$, which is what we wished to show.
\end{proof}

Evidently, combining Lemmas~\ref{mainevench} and \ref{Lpols}, and
Propositions~\ref{2mod4} and \ref{0mod4}
one obtains a proof of Theorem~\ref{evencheb}.

\vspace{\baselineskip}
We will now discuss Conjecture~\ref{conj}.
To this end, we first describe an attempt
to prove the conjecture which unfortunately
seems to fail.
\begin{remark}\label{rem4.3}{\rm
We continue with the notations introduced
in the proofs of Theorem~\ref{mainodd}
and Lemma~\ref{Lpols}; in particular, the integer $d>0$ is assumed to be odd. 
A natural way to describe a decomposition of a Jacobian variety such
as $\cJ(\cX)$ is in terms of suitable endomorphisms of this Jacobian.
We refer to the paper of Kani and Rosen \cite{Kani-Rosen} which studies the special
endomorphisms generated by those coming from automorphisms of the curve.

Consider the action of $1+\tau$ and of
$1+\rho^4+\rho^8+\ldots+\rho^{4d-4}$ on $\cJ(\cX)$.
As endomorphisms on $\cJ(\cX)$ these maps are
defined over the prime field of $\FF_q$.
Moreover since $1+\tau$ acts as $0$ on the
regular differentials on $\cX$ which are pulled back
from $\cC$, and as multiplication by $2$ on the
regular differentials pulled back from $\cC_1$,
it follows that $(1+\tau)(\cJ(\cX))$ is isogenous
to $\cJ(\cC_1)$.
An analogous argument shows that
\[
(1+\rho^4+\ldots+\rho^{4d-4})(1+\tau)(\cJ(\cX))
\]
is isogenous to the elliptic curve $\cE$.
Since $1+\rho^4+\ldots+\rho^{4d-4}$ acts as
multiplication by $d$ on the differential
$\omega_{(d+1)/2}$ and as $0$ on the
differentials $\omega_j+\omega_{d+1-j}$ ($1\leq j\leq (d-1)/2$),
it follows that the abelian variety $\cA\subset \cJ(\cX)$
defined by
\[\cA:= (-d+1+\rho^4+\ldots+\rho^{4d-4})(1+\tau)(\cJ(\cX))\]
is defined over the prime field of $\FF_q$, and
$\dim(\cA)=(d-1)/2$, and $\cJ(\cC_1)\sim\cE\times \cA$
(an isogeny defined over the prime field of $\FF_q$).
As a result,
\[\cJ(\cX)\sim \cJ(\cC)\times\cE\times \cA.\]

Suppose that we would know that $\cA$ and $\cJ(\cC)$
are isogenous over $\FF_{q^2}$. Then in particular
the $L$-polynomial $L_{\cC}(t)$
divides $L_{\cC_1}(t)$ (here we take $L$-polynomials
over $\FF_{q^2}$).
Clearly, this would imply the case $d$ odd of
Conjecture~\ref{conj}.

 A rather natural idea for showing that
 indeed the
abelian varieties $\cA$ and $\cJ(\cC)$  are isogenous
over $\FF_{q^2}$, is to look for endomorphisms in
the subalgebra $\ZZ[\rho,\tau]\subset\mbox{End}(\cJ(\cX))$ and restrict those to $\cA$ or to
$(1-\tau)(\cJ(\cX))\sim \cJ(\cC)$. Unfortunately, this cannot work,
as is seen by the following argument.

Consider the regular differentials on $\cX$ that
correspond to $\cA$ and to $\cJ(\cC)$.
The action of $\ZZ[\rho,\tau]$ on the regular
differentials on $\cX$ has the invariant
subspaces $V_j$ spanned by $\omega_j$ and $\omega_{d+1-j}$. If $d>1$ then $\dim(V_1)=2$ and $\tau,\rho$ act
on $V_1$ by the matrices
${{0\;\;-1}\choose{-1\;\;\;0}}$ and
${{\zeta\;\;\;\;\;0}\choose{0\;-\zeta^{-1}}}$, respectively.
We look for an element in the $\ZZ$-algebra generated
by these two matrices that sends one of the two
lines spanned by ${1}\choose{1}$ or by ${1}\choose{-1}$,
to the other.
However, such an element does not exist.
}\end{remark}

\begin{remark}\label{kohelsmithresult}
{\rm In fact Conjecture~\ref{conj} is true in the case that $d$ is (an odd) prime.
Namely, as a special case of Proposition~14 in the paper \cite{KS} by Kohel and Smith,
one obtains that $\cJ(\cX)$ is isogenous to $\mathcal{E}\times \cJ(\cC)\times \cJ(\cC)$
with $\mathcal{E}$ the elliptic curve given by $y^2=x^3-x$ and
$\cC\colon y^2=\varphi_d(x)$. This means that the $L$-polynomial of $\cX$ over
$\FF_{q^2}$ is the product of that of $\mathcal{E}$ and two copies of that of $\cC$. 

As we saw in the proof of Theorem~\ref{evencheb}, the $L$-polynomial of $\cX$
is also the product of that of $\cC$ and that of $\cC_1\colon y^2=(x^2-4)\varphi_d(x)$.
Combining the two factorizations, one concludes that for $d>2$ prime, the $L$-polynomial
of $\cC_1$ equals the product of that of $\cC$ and that of $\mathcal{E}$.
This shows Conjecture~\ref{conj} in this case.
And so by Tate's classical work \cite{Tate} we have that $\cJ(\cC_1)$ is isogenous to $\mathcal{E}\times \cJ(\cC)$.

A natural approach to proving Conjecture~\ref{conj} would be, to show that also for
{\em composite} odd $d$ one has an isogeny $\cJ(\cX)\sim \mathcal{E}\times \cJ(\cC)\times \cJ(\cC)$
defined over $\FF_{q^2}$. Although we have not been able to show this, we can in fact
prove the weaker statement that these abelian varieties are isogenous over the
algebraic closure $\overline{\FF_q}$. Indeed, consider the subgroup $G$ of
$\mbox{Aut}(\cX)$ generated by $r:=\rho^4$ and $s:=\tau$. Then $r$ has order $d$
and $s$ has order $2$. Moreover $srs=r^{-1}$, so $G$ is a dihedral group of order $2d$ (and in the
case considered here, $d$ is odd). 

Following Paulhus \cite[\S~3.1.2]{Paulhus}, who applies Kani-Rosen theory (specifically, \cite[Theorem~B]{Kani-Rosen})
to the subgroups $H_1=\langle r\rangle$ and $H_j=\langle sr^j \rangle$ ($2\leq j\leq 2d+1$) of $G$,
and who observes that because $d$ is odd, all groups $H_j$ ($j\neq 1$) are conjugate in $G$ and therefore
the quotients $\cX/H_j$ are isomorphic, one concludes
\[
\cJ(\cX)\times \cJ(\cX/G) \times \cJ(\cX/G) \sim \cJ(\cX/\langle r\rangle)\times \cJ(\cX/\langle s\rangle\times \cJ(\cX/\langle s\rangle).
\]
We analyze the quotient curves appearing here.
As we saw in the proof of Theorem~\ref{mainodd}, $\cX/\langle s\rangle=\cX/\langle \tau\rangle\cong \cC$
since $d$ is odd.
Moreover, the proof of Lemma~\ref{Lpols} shows $\cX/\langle r\rangle=\cX/\langle\rho^4\rangle\cong \cE$,
and up to scalars, $x^{(d-1)/2}dx/y$ is the only regular differential on $\cX$ invariant under
the action of $r$. As this differential is not fixed by $s=\tau$, no regular differentials fixed by
every automorphism in the group $G$ exist. Therefore the genus of $\cX/G$ equals $0$, so
$\cJ(\cX/G)=(0)$. So the displayed isogeny in fact reads
\[  \cJ(\cX)\sim \cE \times \cJ(\cC)\times \cJ(\cC),
\]
which is what we wished to show. Adapting this line of reasoning so that it works over $\FF_{q^2}$ as well,
would lead to a proof of Conjecture~\ref{conj} but unfortunately, so far we have not been able
to do so.
}\end{remark}

\begin{remark}
\rm{Let $d>0$ be any integer, and let $q$
be a prime power with $\gcd(q,2d)=1$. If $3|d$, then $v^2= \varphi_{d}(u)$ covers the elliptic
curve $v^2= \varphi_{3}(u)=u^3-3u$ since in this case (see Remark~\ref{rem1.2}) we have
  $\varphi_{d}(u)=\varphi_{3}(\varphi_{d/3}(u))$. Hence if
in this case the curve given by $v^2= \varphi_d(u)$ is maximal over $\FF_{q^2}$, then the elliptic curve
 $v^2= \varphi_{3}(u)$ is also maximal over $\FF_{q^2}$.
The latter maximality occurs precisely
when $q \equiv -1(\bmod~4)$.
As a consequence, for $d$ a multiple of $3$
the assumption $q\equiv -1(\bmod~4)$
mentioned in statement (ii) of
Conjecture~\ref{conj} can be deleted.
}\end{remark}

\begin{remark}\label{cheb3mod4}{\rm
In \cite{TTV}, the curve $\cC\colon y^2=\varphi_d(x)$ is denoted by $\cC_0$; one of
the results of that paper (\cite[Section~3.2]{TTV})
is that in case $d=\ell$ is an odd prime number,
then the endomorphism algebra of $\cJ(\cC)$ contains
the field $K:=\QQ(\sqrt{-1},\zeta_\ell+\zeta_\ell^{-1})$.
Note that $[K:\QQ]=\ell-1=2g$ where $g$ is the genus of $\cC$.
Moreover, provided $\ell\neq 5$, regarding
$\cJ(\cC)$ as an abelian variety in characteristic~$0$,
by \cite[Proposition~5]{TTV} it has no nontrivial abelian subvarieties (over any
field extension). This means that $\cJ(\cC)$
is a so-called CM abelian variety. The extension
$K/\QQ$ is Galois (even abelian), with Galois
group $G\cong \ZZ/2\ZZ\times \FF_\ell^\times/\pm 1$;
note that this group is cyclic precisely when
$\ell\equiv 3(\bmod~4)$.

The CM type corresponding to $\cJ(\cC)$ is computed
in \cite{TTV}. One identifies it with the
subset $\Phi\subset G$ given by
\[
\Phi=\left\{(0,\pm 1),\,(1,\pm 2),\,(0,\pm 3), \ldots\right\}
\]
of cardinality $(\ell-1)/2$.

In \cite[Theorem~3.1]{Blake} it is explained how
the slopes of the Newton polygon of Frobenius on a
reduction of $\cC$ modulo a prime $p$ can be
determined from the decomposition group $D\subset G$
at $p$: the possible slopes are $\#(Dg\cap \Phi)/\#Dg$
with $g$ an element of $G$.
Note that the group $D$ (at any prime $p$
with $\gcd(p,2\ell)=1$ which means, at any
prime that does not ramify in $K$) is
the cyclic group generated by
$\left((p-1)/2(\bmod~2), \pm p(\bmod~\ell)\right)$.
In particular, taking $p\equiv 1(\bmod~4)$ one has
that $D\subset (0)\times \FF_\ell^\times/\pm 1$.
Hence taking $g=(1,\pm 1)\in G$ one finds $Dg\cap \Phi=
\emptyset$. As a result, one of the slopes is $0$,
implying that the $p$-rank of $\cJ(\cC)$ is positive.
In particular, this provides an alternative
proof of Proposition~\ref{p2.3} for the
special case of the polynomial $\varphi_d(x)$
with $d>1$ odd. Indeed,
taking $\ell$ any prime divisor of $d$, the
equality $\varphi_d(x)=\varphi_\ell(\varphi_{d/\ell}(x))$ implies that the curve
with equation $y^2=\varphi_\ell(x)$
is covered by the curve given by
 $y^2=\varphi_d(x)$.
Hence if the latter curve is maximal over $\FF_{q^2}$ (and $\gcd(q,2\ell)=1$),
then so is the first, and therefore
the characteristic of $\FF_{q^2}$ is
$\equiv 3(\bmod~4)$.

\vspace*{.4cm}
We illustrate the use of CM theory
also in the next result.
}\end{remark}
\begin{pro}\label{CMresult}
Let $q$ be a prime power
with $\gcd(q,10)=1$. If the hyperelliptic curve given by
 $y^2=x^5-5x^3+5x$ is maximal over $\FF_{q^2}$, then
the characteristic of $\FF_q$ is either
$11(\bmod~20)$ or $19(\bmod~20)$.
\end{pro}
\begin{proof}
Note that $x^5-5x^3+5x=\varphi_5(x)$.
We will show the result by using the CM theory
described above in Remark~\ref{cheb3mod4}.
We therefore use the notations introduced
in that remark, for the special case $\ell=5$.

Let $p$ be the characteristic of $\FF_q$.
By Proposition~\ref{p2.3} (and
alternatively, by Remark~\ref{cheb3mod4}),
maximality of the given curve $\cC$ implies
that $p\equiv 3(\bmod~4)$.
Hence the decomposition group $D$ at p
in $G=\mbox{Gal}(K/\QQ)\cong \ZZ/2\ZZ\times\FF_5^\times/\pm1$ is generated by
$\left(1,\pm p(\bmod 5)\right)$.

In case $p\equiv\pm 2(\bmod~5)$, this means
\[D=\{(1,\pm 2),\,(0,\pm 1)\}.\]
Clearly $D\cdot (0,\pm 2)\cap\Phi=\emptyset$
where $\Phi\subset G$ describes the
CM=type of the curve $\cC$.
As before, this implies that $\cC$
cannot be maximal in characteristic $p$.

So a necessarily condition for maximality
in characteristic $p$ is besides
$p\equiv 3(\bmod~4)$ that also $p\equiv \pm 1
(\bmod~5)$. From this, the result follows.
\end{proof}

\begin{remark}{\rm
In the proof above we only used the fact
that for a maximal curve, the slopes of
Frobenius are all positive. A stronger
condition is that in fact they need to be
equal to $\frac12$. Exploiting that, one
obtains similar results for other
values of $\ell$. For example, with
$\ell=17$ one can exclude characteristic
$p\equiv \pm 2(\bmod~17)$ in this way.
}\end{remark}

We finish this manuscript by briefly
mentioning some small cases of Conjecture~\ref{conj}.
\begin{itemize}
\item[$d=1$:] here statement~(i)
asserts the maximality of the elliptic curve
given by $y^2=x^3-4x$ over $\FF_{q^2}$.
This holds precisely when $q\equiv -1(\bmod~4)$.
Statement~(ii) asserts, besides this congruence
condition, also the maximality of the
hyperelliptic curve given by $v^2=u$. Since
this maximality holds over any $\FF_{q^2}$
(the curve has genus~$0$),
Conjecture~\ref{conj} holds for $d=1$.
\item[$d \geq 5$:] we  verified using Magma for
all  prime powers $q<100$  and $d \in \{9,15,21\}$ Conjecture~\ref{conj} holds.
In fact, the experiment shows for these cases,
as we saw in Remark~\ref{kohelsmithresult} for the case $d$ is an odd prime,
that the curves $\cC\colon y^2=\varphi_d(x)$ and  $\cC_1\colon y^2=(x^2-4)\varphi_d(x)$
over $\FF_{q}$ are related by $\cJ(\cC_1)\sim \cJ(\cC)\times \cE$ with $\cE\colon y^2=x^3+x$.
\end{itemize}

 \section*{Acknowledgement}
{The first author was supported by
FAPESP/SP-Brazil grant 2017/19190-5.
The second author thanks Marco Streng and Nurdag\"{u}l
Anbar Meidl for helpful suggestions. We also thank the referee for
various comments which helped us improve the exposition.}

 \end{document}